\documentclass[a4, 10pt,leqno]{amsart}
\usepackage{pifont}
\usepackage{mathtools}
\usepackage{mathabx}
\usepackage{bbm}
\usepackage{amsthm}
\usepackage{mathtools}
\usepackage{mathtools,accents}
\usepackage{amssymb,amscd,amsthm,enumerate}
\usepackage{mathrsfs}
\usepackage{a4wide}
\usepackage{eqnarray}
\usepackage{enumitem}
\usepackage[utf8]{inputenc}
\usepackage[T1]{fontenc}
\usepackage{lmodern}
\usepackage[english]{babel}
\usepackage{microtype}
\usepackage{float}
\usepackage{esint}
\usepackage{physics}
\usepackage{a4wide}
\usepackage{algorithmic}
\usepackage[ruled]{algorithm2e}
\usepackage[active]{srcltx}
\usepackage{lipsum}
\usepackage{upgreek}
\usepackage{listings}
\usepackage{wasysym}
\usepackage{csquotes}
\usepackage{enumerate}
\usepackage{color}
\usepackage[x11names,usenames,dvipsnames]{xcolor}
\usepackage{xcolor}
\definecolor{ultrablue}{rgb}{0.0,0.0, 1}
\definecolor{jigari}{rgb}{0.39,0.0, 0.0}
\usepackage{url}
\usepackage{hyperref}
\hypersetup{
	linktoc=page,
	colorlinks,
	linkcolor={Blue1},
	citecolor={Blue1},
	urlcolor={Blue1}
}
\hypersetup{breaklinks=true}
\def\house{\hbox{\kern3pt \vbox to13pt{}%
		\pdfliteral{q 0 0 m 0 5 l 5 10 l 10 5 l 10 0 l 7 0 l 7 5 l 3 5 l 3 0 l f
			1 j 1 J -2 5 m 5 12 l 12 5 l S Q }%
		\kern 13pt}}
\usepackage{fontawesome5}	
\usepackage{soul}
\usepackage{graphicx}
\usepackage{tikz}
\usepackage{tikzpagenodes}
\usepackage{scrlayer-scrpage}
\usepackage{tikz}
\usetikzlibrary{calc,spy}
\usetikzlibrary{shapes.geometric, arrows}
\usepackage{pgfplots}
\usetikzlibrary{intersections, pgfplots.fillbetween}
\usepackage{tikz}
\usetikzlibrary{matrix}
\usepackage[all]{xy}
\usepackage{subfig}
\allowdisplaybreaks
\numberwithin{equation}{section}

\theoremstyle{plain}
\newtheorem*{theorem*}{Main Theorem}
\newtheorem*{thm*}{Theorem}
\newtheorem*{corollary*}{Corrolary}
\newtheorem{lemma}{Lemma}[section]

\newtheorem{proposition}[lemma]{Proposition}

\theoremstyle{definition}
\newtheorem{definition}[lemma]{Definition}
\newtheorem{remark}[lemma]{Remark}

\theoremstyle{remark}

\newcounter{example}

\makeatletter
\def\@xnamedef#1{\expandafter\protected@xdef\csname #1\endcsname}
\def\no@harm{} 
\def\ead@au#1{\protected@edef\@ead@au{#1}}
\patchcmd\runningauthor@fmt{\global\edef}{\protected@xdef}{}{}
\patchcmd\runningauthor@fmt{\global\edef}{\protected@xdef}{}{}
\patchcmd\author@fmt{\edef}{\protected@edef}{}{}
\patchcmd\add@xtok{\xdef}{\protected@xdef}{}{}
\makeatother

\renewcommand{\setminus}{\smallsetminus}

\newcommand{\R}{\mathbb R}
\newcommand{\Sp}{\mathbb S}

\DeclareMathOperator{\supp}{supp}

\DeclareMathOperator{\scal}{scal}

\newcommand{\m}{\mathdutchcal{m}}
\newcommand{\y}{{\footnotesize \textbf{\textit{y}}}}

\newcommand{\dist}{\mathrm{d}}

\usepackage{accents}

\makeatletter
\newcommand*\bigcdot{\mathpalette\bigcdot@{.5}}
\newcommand*\bigcdot@[2]{\mathbin{\vcenter{\hbox{\scalebox{#2}{$\m@th#1\bullet$}}}}}
\makeatother
\newcount\stylenum  \newdimen\styledim 
\def\varstyle#1{\mathchoice{\stylenum=0 #1}{\stylenum=1 #1}{\stylenum=2 #1}{\stylenum=3 #1}}
\def\mathaxis{\fontdimen22\ifcase\stylenum 
	\textfont\or\textfont\or\scriptfont\or\scriptscriptfont\fi2 }
\def\setstyledim{\styledim=\ifcase\stylenum .1em\or.1em\or.07em\or.05em\fi\relax}
\def\sqdot{\mathbin{\varstyle{\raise\mathaxis\hbox{\setstyledim
				\kern\styledim 
				\vrule width1.2\styledim height.6\styledim depth.6\styledim
				\kern\styledim}}}}

		\DeclareMathAlphabet{\mathdutchcal}{U}{dutchcal}{m}{n}
		\SetMathAlphabet{\mathdutchcal}{bold}{U}{dutchcal}{b}{n}
		\DeclareMathAlphabet{\mathdutchbcal}{U}{dutchcal}{b}{n}
		\DeclareSymbolFont{myletters}{OML}{ztmcm}{m}{it}
		\DeclareMathSymbol{\nicelambda}{\mathord}{myletters}{"15}
		\usepackage{nicefrac}
\usepackage{stackengine}
\setstackEOL{\\}
\newcounter{tmpctr}
\newcommand\fancyRoman[1]{%
	\setcounter{tmpctr}{#1}%
	\setbox0=\hbox{\kern.2pt\textsf{\Roman{tmpctr}}}%
	\setstackgap{S}{-.6pt}%
	\Shortstack{\rule{\dimexpr\wd0+.1ex}{.7pt}\\\copy0\\
		\rule{\dimexpr\wd0+.1ex}{.7pt}}%
}

\usepackage{graphicx}
\newcommand\smallO{{\scalebox{1}{$\scriptscriptstyle\mathcal{O}$}}}

		\usepackage{pifont}
		
\makeatother

\makeatletter
\newcommand{\tpitchfork}{%
	\vbox{
		\baselineskip\z@skip
		\lineskip-.52ex
		\lineskiplimit\maxdimen
		\m@th
		\ialign{##\crcr\hidewidth\smash{$-$}\hidewidth\crcr$\pitchfork$\crcr}
	}%
}
\makeatletter
\newcommand{\tpmod}[1]{{\@displayfalse\pmod{#1}}}
\makeatother

\usepackage{stackengine,scalerel}

\newcommand\overstar[1]{\ThisStyle{\ensurestackMath{%
			\setbox0=\hbox{$\SavedStyle#1$}%
			\stackengine{0pt}{\copy0}{\kern0\ht0\smash{\SavedStyle\star}}{O}{c}{F}{T}{S}}}}
\makeatletter
\@namedef{subjclassname@2020}{%
	\textup{2020} Mathematics Subject Classification}
\makeatother
\makeatletter 
\def\l@subsection{\@tocline{2}{0pt}{3pc}{6pc}{}}
\makeatother
\makeatletter 
\def\l@subsection{\@tocline{2}{0pt}{3pc}{6pc}{}}
\makeatother
\begin{document}
\title[\scriptsize{Positive mass and rigidity for asymptotically flat tangent bundles}]{\small  Positive mass and rigidity \\ for asymptotically flat tangent bundles} 
%

\author[\protect \scriptsize S. Lakzian]{Sajjad Lakzian}
\address{Department of Mathematical Sciences, Isfahan University of  Technology (IUT), Isfahan 8415683111, Iran.}
\email{\href{mailto:slakzian@iut.ac.ir}{slakzian@iut.ac.ir}}

\address{School of Mathematics, Institute for Research in Fundamental Sciences (IPM), P. O. Box 19395-
	5746, Tehran, Iran}
\email{\href{mailto:lakzians@gmail.com}{lakzians@gmail.com}}
\thanks{This work was partially supported by the Grant No. 1404530313 from IPM}
\thanks{This work is supported by the INSF Grant No. 4030556, awarded by the “On
	the Frontiers of Mathematical Sciences” program.}
\maketitle
%
\begin{abstract}
\par \textsl{Let $(M,g)$ be an asymptotically flat manifold such that its tangent bundle $TM$ is diffeomorphically asymptotically Euclidean. We prove a positive mass theorem and a positive mass rigidity theorem for a natural class of asymptotically flat metrics $\widetilde g$ on $TM$ exhibiting slow asymptotic decay. These metrics are constructed by interpolating, along the horizontal distribution of $TM$, between the Euclidean metric and the base metric $g$ near infinity. The main difficulty is that the induced metrics on $TM$ decay below the standard threshold for the ADM mass in dimension $2n$; nevertheless, we show that their asymptotic mass is well-defined in a generalized sense and is determined by geometric data on the underlying manifold $M$.}
\end{abstract}
\date{\today}
\section{Introduction}\label{sec:intro}
\par The positive mass theorem (PMT) and its rigidity statement are fundamental results in geometric analysis and mathematical relativity. In its classical form, the theorem asserts that an asymptotically Flat manifold with nonnegative integrable scalar curvature has nonnegative ADM mass, and that vanishing mass forces the manifold to be Euclidean.
\par In essence, the PMT and rigidity theorem asserts the physically intuitive principle that, for any asymptotically flat time slice of an isolated gravitational system, the total mass — measured in this setting by the Arnowitt–Deser–Misner (ADM) mass — is nonnegative whenever the local mass density - the role played by the scalar curvature - is nonnegative. Moreover, the total mass vanishes if and only if the manifold is isometric to Euclidean space (spatial vaccum).
\par In this paper, we investigate an analogous question for asymptotically flat metrics on tangent bundles of asymptotically flat manifolds. 
\subsection{ADM mass and the positive mass theorem}\label{subsec:SY}
\par We begin by recalling the relevant features of the ADM mass. For a smooth function $f$ defined on an open subset of $\R^n$, we use the notation
$
f_{,i}:=\nicefrac{\partial f}{\partial x^i}.
$
If $f$ is smooth outside a compact subset of $\R^n$, we write
\[
f\in\mathcal{O}_k\left(|x|^{-\tau}\right)
\]
provided
\[
|f|+|x|\|Df\|+\cdots+|x|^k\left\|D^kf\right\|
=\mathcal{O}\left(|x|^{-\tau}\right)
\]
as $|x|\to\infty$.
\begin{definition}[Asymptotically Flat manifold]\label{defn:AE}
	A complete Riemannian manifold $(M,g)$ is called asymptotically Flat (AF) if there exists a compact set $K\Subset M$ such that
	\[
	M\setminus K=M_1\sqcup\cdots\sqcup M_k,
	\]
	where each end $M_a$ is diffeomorphic to the complement of a closed Euclidean ball in $\mathbb R^n$, and, in the corresponding coordinates near infinity,
	\[
	g_{ij}-\delta_{ij}\in\mathcal O_2\left(|x|^{-\sigma_a}\right)
	\]
	for some $\sigma_a>0$.
\end{definition}
\par We deliberately do not incorporate the usual restrictions on the decay rate and scalar curvature into the definition. Whenever needed, we shall impose the stronger hypotheses explicitly.
\par For an AF end $(M_a,g)$, the ADM mass is given by
\begin{align}\label{eq:ADM}
	\m(M_a,g)
	:=
	\frac{1}{2(n-1)\omega_{n-1}}
	\lim_{r\to\infty}
	\int_{\mathbb S^{n-1}(r)}
	\left(g_{ij,i}-g_{ii,j}\right)\frac{x^j}{r}\,
	d\mathcal H^{n-1},
\end{align}
whenever the limit exists and is coordinate-independent. Here
\[
\omega_{n-1}:=\mathcal H^{n-1}(\mathbb S^{n-1}(1)\subset\mathbb R^n),
\qquad r:=|x|.
\]
For the standard theory of AF manifolds and the ADM mass, see Lee~\cite[Chapter 3]{Lee2019}; see also the original reference Arnowitt--Deser--Misner~\cite{ArnowittDeserMisner1962} and Bartnik~\cite{Bartnik1986}.
\par The ADM mass is a flux quantity at infinity associated with the asymptotic geometry of the metric; indeed it can be interpreted as the flux of the gravitational field across coordinate spheres at infinity; see Ashtekar-Hansen~\cite{AshtekarHansen1978} and Chrúsciel~\cite{Chrusciel1986a}.
\par The coordinate invariance of the ADM mass is closely tied to the decay rate of the metric. When the scalar curvature is integrable and
\[
\sigma_a>\frac{n-2}{2},
\]
the ADM mass is well-defined and independent of the choice of admissible coordinates at infinity. This is the regime in which the classical positive mass theorem applies.
\begin{thm*}[Positive mass theorem and rigidity]\label{thm:PMTR}
	Let $(M^n,g)$ be a complete asymptotically Euclidean manifold with decay rates
	\[
	\sigma_a>\frac{n-2}{2},
	\]
	and suppose that $\scal_g$ is nonnegative and integrable. Then the ADM mass of each end of $M$ is nonnegative. Moreover, if the ADM mass of an end vanishes, then $(M,g)$ is isometric to $(\mathbb R^n,\delta)$.
\end{thm*}
\par The positive mass theorem has been established by several different methods and in a variety of geometric settings; see Schoen--Yau~\cite{SchoenYau1979b,SchoenYau1979a,SchoenYau1981a,SchoenYau2017}, Witten~\cite{Witten1981}, and Lohkamp~\cite{Lohkamp2006,Lohkamp2016}. Additional geometric structures can lead to alternative proofs or further interpretations of the mass; see, for example, Lam~\cite{Lam2011}, Mirandola--Vit\'orio~\cite{MirandolaVitório2015}, and Heine--LeBrun~\cite{HeinLeBrun2016}.
\subsection{The obstruction on tangent bundles}
\par Our goal is to investigate positive mass and rigidity for natural AF metrics on tangent bundles of AF manifolds. The basic difficulty is caused by a simple but fundamental dimensional obstruction.
\par Suppose that $(M^n,g)$ is AF. The critical decay rate for the ADM mass in dimension $n$ is
\[
\sigma_{\sf crit}=\frac{n-2}{2}.
\]
At decay rates below this threshold, the flux defining the ADM mass need not be coordinate-independent. At the other end, decay faster than $|x|^{-(n-2)}$ forces the ADM mass to vanish in the standard setting. Thus the rate
\[
\sigma=n-2
\]
is the upper range in which a nontrivial ADM mass can occur. Also the rate $(n-2)$ is natural in the sense that the flux defining the ADM mass admits a coordinate-free formula; see \hyperref[subsubsec:interpolating]{Section~\ref{subsubsec:interpolating}}.
\par Now consider the tangent bundle $TM$, which has dimension $2n$. The corresponding critical decay rate is
\[
\lambda_{\sf crit}=\frac{2n-2}{2}= n-1.
\]
Consequently,
\[
n-2<n-1=\lambda_{\sf crit}.
\]
This inequality is the fundamental source of the difficulty in the present paper. Even when the base metric has the natural decay rate $\sigma=n-2$, a metric on $TM$ naturally constructed from $g$ can have decay at or below the critical rate for the $2n$-dimensional ADM mass. Thus, the classical ADM theory is not directly applicable to the natural tangent-bundle metrics of interest here.
\par This phenomenon forces us to work in a regime in which the usual ADM mass may fail to exist as a coordinate-independent limit. Rather than imposing additional decay assumptions that would exclude the natural tangent-bundle geometry, we retain the asymptotic flux information through its lower and upper limits.
\subsection{Lower and upper ADM masses}\label{subsec:ADM}
\par Let $(M,g)$ be an AF manifold, let $M_a$ be an end, and fix coordinates at infinity
\[
\varphi_\infty=\left(x^1,\ldots,x^n\right).
\]
We define the \emph{lower ADM mass} by
\begin{align*}
	\underline{\m}\left(M_a,g,\varphi_\infty\right)
	:=
	\frac{1}{2(n-1)\omega_{n-1}}
	\varliminf_{r\to\infty}
	\int_{\mathbb S^{n-1}(r)}
	\left(g_{ij,i}-g_{ii,j}\right)
	\frac{x^j}{r}\,
	d\mathcal H^{n-1},
\end{align*}
and the \emph{upper ADM mass} by
\begin{align*}
	\overline{\m}\left(M_a,g,\varphi_\infty\right)
	:=
	\frac{1}{2(n-1)\omega_{n-1}}
	\varlimsup_{r\to\infty}
	\int_{\mathbb S^{n-1}(r)}
	\left(g_{ij,i}-g_{ii,j}\right)
	\frac{x^j}{r}\,
	d\mathcal H^{n-1}.
\end{align*}
\par When the classical ADM mass exists, these quantities agree with it. More precisely, under the usual decay and integrability assumptions,
\[
\sigma_a>\frac{n-2}{2}
\quad\Longrightarrow\quad
\underline{\m}\left(M_a,g,\varphi_\infty\right)
=\overline{\m}\left(M_a,g,\varphi_\infty\right)
=\m(M_a,g).
\]
In particular, the classical mass vanishes when the metric decays faster than $|x|^{-(n-2)}$. At decay rates
\[
\sigma_a\leq\frac{n-2}{2},
\]
the ADM flux need not converge and may depend on the coordinates at infinity; see, for example, Denisov--Solov{\textquotesingle}ev~\cite{DenisovSolov1983}. The lower and upper masses remain well-defined as extended real numbers for any fixed coordinate system.
\subsection{AF tangent bundles}
\par We now turn to the geometric setting of the paper. Let $(M^n,g)$ be an AF manifold. We are interested in metrics $\widetilde g$ on $TM$ that are geometrically related to $g$, rather than arbitrary metrics on the $2n$-dimensional manifold $TM$.
\subsubsection{\small\sf{\textbf{Diffeomorphic rigidity}}}
\par There is already a strong topological restriction on AF tangent bundles. If both $M$ and $TM$ are AF, then $M$ is necessarily open and contractible. In particular, $TM$ is diffeomorphic to $\mathbb R^{2n}$, and both $M$ and $TM$ have a single end. We discuss these restrictions in greater detail in \hyperref[sec:AETB]{Section~\ref{sec:AETB}}.
\subsubsection{\small\sf{\textbf{AF tangent bundles and preferred coordinates}}}\label{subsubsec:dimensionality}
\par We shall use coordinates on $TM$ that are induced, in a suitable sense, by coordinates at infinity on the base manifold. 
\begin{definition}[AF tangent bundles with decay $\geq\lambda$]\label{defn:AE-TB}
	Let $(M^n,g)$ be a contractible AF manifold with decay rate
	\[
	\sigma>\frac{n-2}{2}.
	\]
	For a given metric $\widetilde g$ on $TM$, the pair $(TM,\widetilde g)$ is said to be an \emph{AF tangent bundle} of $(M,g)$ with decay $\geq\lambda>0$ if there exist coordinates at infinity 
	\[
	\left(x^1,\ldots,x^n\right)
	\]
	on $M$ such that the induced tangent coordinates
	\[
	(x^1,\ldots,x^n,y^1,\ldots,y^n)
	\]
	extend to coordinates
	\[
	(\zeta^1,\ldots,\zeta^{2n})
	\]
	at infinity on $TM$ in which $\widetilde g$ is AF with decay at least $\lambda$.
\end{definition}
\par We call any such coordinate system
\[
(\zeta^1,\ldots,\zeta^{2n})
\]
a \emph{preferred coordinate system} near infinity, and denote the corresponding coordinate chart by $\zeta$.
\subsubsection{\small\sf{\textbf{Metrics induced by the base geometry}}}\label{subsubsec:interpolating}
\par The tangent bundle carries additional geometric structure that allows one to relate its metric geometry to that of the base manifold. We exploit this structure by restricting attention to AF metrics $\widetilde g$ that are canonically related to $g$.
\par The most natural component to prescribe from the base metric is the horizontal component. Indeed, horizontal tangent vectors on $TM$ are naturally identified with tangent vectors on $M$, once a connection has been chosen. This leads to a class of metrics whose horizontal geometry interpolates between $g$ and the Euclidean metric near infinity.
\par The positive mass theorem proved below applies to a class of metrics that we call \emph{admissible metrics}; see \hyperref[subsubsec:adms-met]{Section~\ref{subsubsec:adms-met}}. Roughly, an admissible metric is Euclidean outside a uniform tubular neighbourhood of the zero section and has, near infinity, the form
\begin{align*}
	\begin{cases}
		\widetilde g_{(x,\y)}(X^h,Y^h)
		=
		\bigl(\psi(\y)g_p+(1-\psi(\y))\delta\bigr)(X,Y),\\[0.4em]
		\widetilde g_{(x,\y)}(X^h,Y^v)= a^c_{(x,\y)}(X,Y),\\[0.4em]
		\widetilde g_{(x,\y)}(X^v,Y^v)= \bigl(\xi(x,\y)h_p+(1-\xi(x,\y))\delta\bigr)(X,Y),
	\end{cases}
\end{align*}
where
\[
0\leq\psi,\xi\leq1,
\qquad
\psi(\y),\xi(x,\y)\longrightarrow0
\quad\text{as }|\y|\to\infty,
\]
and
\[
a^c_{(x,\y)}\longrightarrow0
\quad\text{as }|\y|\to\infty.
\]
Here $a^c$ denotes the mixed, or cross, term; $h$ is an auxiliary Riemannian metric on the base; and the horizontal and vertical lifts are defined with respect to a fixed background metric $\mathdutchcal g$, which we refer to as the \emph{connecting metric}. The precise decay assumptions and the full definition of admissibility are given in \hyperref[subsubsec:adms-met]{Section~\ref{subsubsec:adms-met}}.
\par The base metric will be assumed to satisfy the decay condition
\[
\sigma\geq n-2.
\]
This is not merely a convenient technical choice. The decay rate $\sigma=n-2$ plays a distinguished role in the theory of asymptotically flat metrics. In particular, decay faster than $|x|^{-(n-2)}$ forces the ADM mass to vanish in the standard setting, while at the rate $n-2$ the mass admits an intrinsic formulation; see Lee~\cite[Remark 3.16]{Lee2019}. 
\par \emph{This condition is essentially optimal for our approach: when the decay rate is below $n-2$, the relevant asymptotic flux on the tangent bundle need not remain bounded}; see \hyperref[rem:opt-decay]{Remark~\ref{rem:opt-decay}}.
\subsection{Why tangent bundles?}
\par The geometric construction considered in this work is also reminiscent of the philosophy underlying Kaluza–Klein theories, where additional geometric directions are introduced in order to encode and study physical fields through the geometry of a higher-dimensional space. In the present setting, the tangent bundle $TM$ provides a canonical geometric enlargement of an asymptotically flat manifold $(M,g)$, with the horizontal directions carrying information inherited from the gravitational geometry of $M$ and the vertical directions encoding the infinitesimal structure of the manifold itself. $TM$ can also be interpreted as the phase space of a moving particle on $M$. 
\par Unlike classical Kaluza–Klein constructions, where the extra dimensions are typically compact and introduced to generate gauge fields, the fibers of $TM$ are noncompact and arise naturally from the differential geometry of the underlying space. Our analysis may therefore be viewed as a Riemannian analogue of a higher-dimensional gravitational extension: we investigate whether fundamental properties of gravitational energy, namely positivity and rigidity of the ADM mass, persist under this canonical passage from $M$ to $TM$. 
\subsection{Main results}
\par The dimensional obstruction described above places the natural metrics on $TM$ outside the classical regime for the $2n$-dimensional ADM mass. Nevertheless, the additional structure inherited from the base manifold allows us to recover a positive mass theorem in this slow-decay setting.
\par The main result establishes positivity and rigidity for the lower ADM mass of an admissible metric on $TM$. The theorem also gives, under an additional trace-rigidity hypothesis on the base metric, an explicit relation between the mass of the tangent bundle and the ADM mass of the base.
\par To formulate the result, we need one additional condition on the asymptotic geometry of $g$. Since
\[
\tr_0g\longrightarrow n
\qquad\text{as }|x|\longrightarrow\infty,
\]
we require this convergence to occur, in an averaged sense, from above.
\begin{definition}[Weakly trace-majorizing]
	An AF metric $g$ on $M$ is called \emph{weakly trace-majorizing} if
	\[
	\fint_{\mathbb S_r}\tr_0g\,d\sigma_r\geq n
	\]
	for all sufficiently large $r$. If
	\[
	\fint_{\mathbb S_r}\tr_0g\,d\sigma_r=n
	\]
	for all sufficiently large $r$, we call $g$ \emph{weakly trace-rigid}.
\end{definition}
\begin{theorem*}\label{thm:main}
	Let $(M,g)$ be a contractible AF manifold. Suppose that $g$ is weakly trace-majorizing, has decay rate
	\[
	\sigma\geq n-2,
	\]
	and has nonnegative integrable scalar curvature. Let $\widetilde g$ be an admissible metric on $TM$. Then the connecting metric $\mathdutchcal g$ is flat and
	\begin{enumerate}
		\item[\small \textbf{\textsf{(a)}}] \emph{\small \textsf{\textbf{Positivity of mass.}}}
		\[
		0\leq
		\underline{\m}(TM,\widetilde g,\zeta)
		\leq
		\overline{\m}(TM,\widetilde g,\zeta)
		<\infty,
		\]
		for any set of preffered coordinates $\zeta$.
		If, in addition, $g$ is weakly trace-rigid, then the ADM mass of $\left(TM, \widetilde g\right)$ in the preffered coordinates exist and
		\[
		 \m(TM,\widetilde g, \zeta) =
		\frac{(n-1)\omega_{n-1}}
		{(2n-1)\omega_{2n-1}}
		\m(M,g).
		\]
		\item[\small \textbf{\textsf{(b)}}] \emph{\small \textsf{\textbf{Rigidity.}}}
		If in a set of preferred coordinates,
		\[
		\underline{\m}(TM,\widetilde g,\zeta)=0,
		\]
		then
		\[
		(M,g)\cong_{\mathrm{isom}}(\mathbb R^n,\delta).
		\]
	\end{enumerate}
\end{theorem*}
\par Thus, although the classical ADM mass need not exist for the natural metrics considered on $TM$, the lower and upper asymptotic fluxes remain finite and nonnegative. More significantly, vanishing of the lower mass of the $2n$-dimensional tangent-bundle metric forces the original $n$-dimensional manifold to be Euclidean. In the weakly trace-rigid case, the tangent-bundle mass is quantitatively determined by the ADM mass of the base.
\par A particularly natural subclass is given by the \emph{good metrics}, namely admissible metrics for which the mixed term vanishes,
\[
a^c\equiv0,
\]
and for which $h=g$. In this case the tangent-bundle metric is decoupled in its horizontal and vertical components.
\begin{corollary*}
	In the setting of the \hyperref[thm:main]{Main Theorem}:
	\begin{enumerate}
		\item 	If $\widetilde g$ is a good metric. Then
		\[
		0\leq
		\underline{\m}(TM,\widetilde g,\zeta)
		\leq
		\overline{\m}(TM,\widetilde g,\zeta)
		<\infty.
		\]
		Moreover,
		\[
		\underline{\m}(TM,\widetilde g,\zeta)=0
		\]
		if and only if
		\[
		(M,g)\cong_{\mathrm{isom}}(\mathbb R^n,\delta)
		\qquad\text{and}\qquad
		(TM,\widetilde g)\cong_{\mathrm{isom}}(\mathbb R^{2n},\delta).
		\]
		\item If $g$ is weakly trace-rigid. Then 
		\[
		(M,g)\cong_{\mathrm{isom}}(\mathbb R^n,\delta)
		\] 
		if and only if there exists an admissible metric $\widetilde g$ on $TM$ with vanishing lower mass.
	\end{enumerate}
\end{corollary*}
\subsubsection*{\small\sf{\textbf{Organization of the paper}}}
The remainder of the paper is organized as follows. We begin by examining the topology and asymptotic geometry of AF tangent bundles. We then introduce the class of admissible metrics on $TM$ and analyze the associated connecting metric. The proof of the main theorem proceeds by relating the asymptotic mass flux integral for \(\widetilde{g}\) on \(TM\) to geometric quantities on the base manifold \(M\), and ultimately to the ADM mass of \(g\).
\section{The geometric setup}\label{sec:metrics-TM}
We begin by describing the geometric structure on the tangent bundle that will be used throughout the paper. Let
\[
(x^1,\ldots,x^n)
\]
be a local coordinate system on \(M\). Every vector \(\y\in T_xM\) can be written uniquely as
\[
\y=y^i\partial_{x^i},
\qquad
y^i=dx^i(\y).
\]
Thus
\[
(x^1,\ldots,x^n,y^1,\ldots,y^n)
\]
defines a local coordinate system on \(TM\). We use the notation
\[
\partial_i:=\partial_{x^i},
\qquad
\dot\partial_i:=\partial_{y^i}
\]
for the corresponding coordinate vector fields.
\subsection{The horizontal--vertical splitting}
\par Fix a Riemannian metric \(\mathdutchcal g\) on \(M\), which we refer to as the \emph{connecting metric}. Its Levi--Civita connection induces a canonical horizontal--vertical decomposition
\[
TTM=\mathcal H M\oplus\mathcal V M,
\]
where \(\mathcal H M\) and \(\mathcal V M\) denote the horizontal and vertical subbundles of \(TTM\). Each of these bundles is naturally isomorphic to the pullback bundle \(\pi^*TM\).
\par In the coordinates above, the connection coefficients are
\[
N_i^{\,j}(x,\y) = y^b\,{}^{\mathdutchcal g}\Gamma_{ib}^{\,j}(x),
\]
where \({}^{\mathdutchcal g}\Gamma_{ib}^{\,j}\) are the Christoffel symbols of \(\mathdutchcal g\). The associated adapted frame is
\[
\delta_i
:=
\partial_i-N_i^{\,j}\dot\partial_j,
\qquad
\dot\partial_i,
\]
with \(\delta_i\) spanning the horizontal distribution and \(\dot\partial_i\) spanning the vertical distribution.
\par Accordingly, if \(X=X^i\partial_i\in T_xM\), its horizontal and vertical lifts to \(T_{(x,\y)}TM\) are
\[
X^h=X^i\delta_i,
\qquad
X^v=X^i\dot\partial_i.
\]
We will use these lifts without further comment; see, for example, \cite{Sasaki1958}.
\subsection{Riemannian metrics on $\pmb{TM}$}
\par Relative to the horizontal--vertical splitting determined by the connecting metric \(\mathdutchcal g\), every Riemannian metric \(\widetilde g\) on \(TM\) admits a block decomposition
\[
\begin{aligned}
	\widetilde g_{(x,\y)}(X^h,Y^h)
	&=g^h_{(x,\y)}(X,Y),\\
	\widetilde g_{(x,\y)}(X^h,Y^v)
	&=a^c_{(x,\y)}(X,Y),\\
	\widetilde g_{(x,\y)}(X^v,Y^v)
	&=g^v_{(x,\y)}(X,Y).
\end{aligned}
\]
Here, for each fixed \(\y\), the tensors $g^h_{(\cdot,\y)}$ and $g^v_{(\cdot,\y)}$ are Riemannian metrics on \(M\), while
$
a^c_{(\cdot,\y)}
$
is a symmetric \(2\)-tensor; the superscript \(c\) indicates the cross term.
\par Thus, once the connecting metric \(\mathdutchcal g\) is fixed, we regard any such  \(\widetilde g\)  as the quadruple
\[
\widetilde g = \widetilde g(\mathdutchcal g,g^h,a^c,g^v),
\]
eventhough, $\mathdutchcal g,g^h,a^c,g^v$ only uniquenly determine $\widetilde{g}$ on $TU_\infty$. 
\par The distinction between the connecting metric and the metric \(\widetilde g\) on \(TM\) will be important below. In particular, the horizontal--vertical splitting depends on \(\mathdutchcal g\), whereas the three blocks \(g^h,a^c,g^v\) determine \(\widetilde g\) relative to that splitting.
\subsection{Rigidity of the connecting metric}
\par We next record a rigidity phenomenon for the connecting metric $\mathdutchcal{g}$ that will be used later on.
\subsubsection{\small\sf{\textbf{ Genericity}}}
Let \((U,{x^i})\) be a coordinate chart. We say that a property of smooth functions on \(U\) is \emph{generic} if it holds on an open dense subset of \(\mathcal C^\infty(U)\) endowed with the point-open topology, equivalently, the topology of pointwise convergence. In particular, boundedness is generic in this sense.
\par A Riemannian metric \(g\) on \(U\) will be called \emph{generic} if each of its coordinate coefficients \(g_{ij}\) is generic with respect to the chosen coordinate system. In particular, for every pair of indices \(i,j\),
\[
\left\{x\in U \quad \text{\textbrokenbar}\quad g_{ij}(x)\neq0\right\}
\]
is an open dense subset of \(U\).
\begin{remark}
	This notion of genericity is specific to the present setting and should not be confused with the usual Baire-category notion of genericity in differential topology.
\end{remark}
\begin{lemma}\label{lem:g-gen-flat}
	Let \({x^i,y^i}\) be the coordinates on \(TM\) induced by a coordinate chart \((U,{x^i})\) on \(M\). Suppose that \(\widetilde g\) is generic. Then
	\[
	N_i^{\,k}g^v_{kj}\equiv0,
	\qquad
	\forall,i,j.
	\]
	Consequently,
	\begin{align*}
		\widetilde g_{ij}=g^h_{ij},
		\qquad
		\widetilde g_{i\dot j}=a^c_{ij},
		\qquad
		\widetilde g_{\dot i\dot j}=g^v_{ij}.
	\end{align*}
\end{lemma}
\begin{proof}
	In the coordinate frame,
	\[
	\partial_i = \delta_i+N_i^{\,k}\dot\partial_k.
	\]
	Therefore
	\begin{align}\label{eq:g-tild-horiz}
		\widetilde g_{ij}
		&=
		\widetilde g
		\left(
		\delta_i+N_i^{\,k}\dot\partial_k,
		\delta_j+N_j^{\,l}\dot\partial_l
		\right)\notag\\
		&=
		g^h_{ij}
		+
		N_j^{\,l}a^c_{il}
		+
		N_i^{\,k}a^c_{jk}
		+
		N_i^{\,k}N_j^{\,l}g^v_{kl},
	\end{align}
	while
	\begin{align}\label{eq:g-vert}
		\widetilde g_{i\dot j}
		&=
		\widetilde g
		\left(
		\delta_i+N_i^{\,k}\dot\partial_k,
		\dot\partial_j
		\right)\notag\\
		&=
		a^c_{ij}+N_i^{\,k}g^v_{kj},
	\end{align}
	and
	\[
	\widetilde g_{\dot i\dot j}=g^v_{ij}.
	\]
	\par By genericity, the relevant coordinate coefficients may be taken to be bounded. Since
	\[
	N_i^{\,k}(x,\y) = y^b\,{}^{\mathdutchcal g}\Gamma_{ib}^{\,k}(x),
	\]
	equation \hyperref[eq:g-vert]{\eqref{eq:g-vert}} implies that
	\[
	y^b\,{}^{\mathdutchcal g}\Gamma_{ib}^{\,k}(x)
	\,g^v_{kj}(x,\y)
	\]
	is bounded as \(|\y|\to\infty\). Replacing \(\y\) by \(\lambda\y\) and using the boundedness of \(g^v_{kj}\), we obtain
	\[
	{}^{\mathdutchcal g}\Gamma_{ib}^{\,k}(x)
	\,g^v_{kj}(x,\y)=0, \quad \forall i,j
	\]
	for every \(\y\neq0\). By continuity, this holds for all \((x,\y)\), and hence
	\[
	N_i^{\,k}g^v_{kj}\equiv0.
	\]
	\par Substituting this identity into \hyperref[eq:g-tild-horiz]{\eqref{eq:g-tild-horiz}} gives
	\[
	\widetilde g_{ij} = g^h_{ij}
	+
	N_j^{\,l}a^c_{il}
	+
	N_i^{\,k}a^c_{jk}.
	\]
	Again using boundedness and the fact that \(N_i^{\,k}\) is linear in \(\y\), we may replace \(\y\) by \(\lambda\y\) and let \(\lambda\to\infty\). This yields
	\[
	N_j^{\,l}a^c_{il}
	+
	N_i^{\,k}a^c_{jk}
	\equiv0.
	\]
	Consequently,
	\[
	\widetilde g_{ij}=g^h_{ij},
	\qquad
	\widetilde g_{i\dot j}=a^c_{ij},
	\qquad
	\widetilde g_{\dot i\dot j}=g^v_{ij},
	\]
	as claimed.
\end{proof}
\subsection{Asymptotically Euclidean tangent bundles}\label{sec:AETB}
\par We now turn to the topology and asymptotic geometry of tangent bundles. We first recall the purely topological (and/or differential) notions of asymptotic Euclideanity that will be needed.
\begin{definition}[Top-AE and diff-AE manifolds]
	A manifold \(M^n\) is called \emph{topologically asymptotically Euclidean} (top-AE), respectively \emph{differentially asymptotically Euclidean} (diff-AE), if there exists a compact set \(K\Subset M\) such that
	\[
	M\setminus K=M_1\sqcup\cdots\sqcup M_k,
	\]
	where each end \(M_a\) is homeomorphic, respectively diffeomorphic, to
	\[
	\mathbb R^n\setminus\overline{B_{R_a}}
	\]
	for some \(R_a>0\).
	\par Moreover, the corresponding homeomorphisms or diffeomorphisms are assumed to be restrictions of maps
	\[
	\varphi_a:M_a\longrightarrow\mathbb R^n
	\]
	whose expression in polar coordinates is
	\[
	\varphi_a(p)=(r,q),
	\]
	with \(r\) increasing monotonically toward infinity along the end. More precisely, if
	\[
	K_1\Subset U\subset K_2,
	\]
	where \(U\subset M_a\) is open, then every point of \(M_a\setminus K_2\) has strictly larger \(r\)-coordinate than every point of \(U\setminus K_1\). See \cite[Definition 1.1]{AsadiFathiLakzian2025}.
\end{definition}
\par The distinction between the notions top-AE and diff-AE will therefore only concern the underlying topology and smooth structure.
\begin{remark}
	Every asymptotically flat Riemannian manifold is, in particular, both topologically and differentially asymptotically Euclidean.
\end{remark}
\par The topology of an AE tangent bundle is strongly constrained. Asadi--Fathi--L. \cite{AsadiFathiLakzian2025} classified vector bundles of rank at least two that are topologically asymptotically Euclidean. Applied to the tangent bundle, their result implies that if \(M\) is complete and noncompact and \(TM\) is topologically asymptotically Euclidean, then \(M\) must be an open contractible manifold.
\par Thus, for an AF manifold \((M,g)\), the additional assumption that \(TM\) be AE forces $M$ to be open and contractible. In particular, the tangent bundle is trivial,
\[
TM\cong M\times\mathbb R^n \cong\mathbb R^{2n}.
\]
When \(n=2\), this is the standard, non-exotic \(\mathbb R^4\).
\par The class of open contractible manifolds other than the Euclidean space is considerably large. In particular, there exist uncountably many pairwise non-homeomorphic open contractible \(n\)-manifolds. Classical examples include contractible open subsets of \(\mathbb R^n\) that are not interiors of compact manifolds with boundary, such as the Whitehead manifold. See, for example, \cite{McMillan1962,CurtisKwun1965,Glaser1967}.
\par We will also need the following one-endedness consequence.
\begin{proposition}
	Suppose that \(TM\) is topologically asymptotically Euclidean. Then \(M\) is contractible and has exactly one topological end. Consequently,
	\[
	TM\cong\mathbb R^{2n}
	\]
	and \(TM\) also has exactly one topological end.
\end{proposition}
\begin{proof}
	By \cite{AsadiFathiLakzian2025}, \(M\) is contractible. It remains to prove that \(M\) has exactly one end.
	\par Let \(K\Subset M\) be compact. Since \(M\) is contractible,
	\[
	\widetilde H^0(M)=\widetilde H^1(M)=0.
	\]
	The long exact sequence of the pair \((M,M\setminus K)\) therefore gives
	\[
	\widetilde H^0(M\setminus K)
	\cong
	H^1(M,M\setminus K).
	\]
	Passing to the direct limit over compact subsets \(K\) as they get larger (limit over the directed system consisting of compact sets with inclusions),
	\[
	\varinjlim_{K \uparrow \infty}\widetilde H^0(M\setminus K)
	\cong
	\varinjlim_{K \uparrow \infty} H^1(M,M\setminus K)
	\cong
	H^1_{\mathrm{cs}}(M).
	\]
	By Poincaré duality for compactly supported cohomology (see, for example, \cite[Theorem 3.35]{Hatcher2002}),
	\[
	H^1_{\mathrm{cs}}(M)
	\cong
	H_{n-1}(M)=0,
	\]
	since \(M\) is contractible. Hence
	\[
	\varinjlim_{K \uparrow \infty}\widetilde H^0(M\setminus K)=0.
	\]
	The number of connected components of the complements of compact sets can be detected using the (reduced) zeroth cohomology; i.e.,
	\[
	\texttt{\#} (\mathrm{ends}) := \varinjlim_{K \uparrow \infty} {H}^0(M\smallsetminus K) = 1
	\]
Thus \(M\) has exactly one end.
	\par Finally, since 
	$
	TM\cong\mathbb R^{2n}
	$, the tangent bundle also has exactly one end.
\end{proof}
\par In what follows, we therefore assume that \(M\) is a complete, contractible, one-ended, differentially asymptotically Euclidean manifold. Consequently, \(TM\) is also a one-ended differentially asymptotically Euclidean manifold.
\subsection{Tubular and interpolating metrics}
\par We next specify the class of metrics on \(TM\) relevant to our mass construction. Let
$
(TM,\widetilde g)
$
be an AF tangent bundle of an AF manifold \((M,g)\), with preferred coordinates
\[
\zeta = \left(\zeta_1,\ldots,\zeta_{2n}\right)
\]
near infinity, as in \hyperref[defn:AE-TB]{Definition~\ref{defn:AE-TB}}.
\par The metrics of interest will be Euclidean away from a tubular neighborhood of the zero section.
\begin{definition}
	We say that \(\widetilde g\) is \emph{tubular} if there exists \(\Lambda>0\) such that, in the preferred coordinates,
	\[
	\supp\bigl(\widetilde g_{ij}-\delta_{ij}\bigr)
	\subset T_\Lambda(M),
	\qquad
	\forall\,i,j,
	\]
	where
	\[
	T_\Lambda(M)
	:=
	\left\{
		(x,\y)\in TM \quad \text{\textbrokenbar}\quad
		\dist_{\widetilde g}\bigl((x,\y),M\bigr)\le\Lambda
		\right\}
	\]
	is the \(\widetilde g\)-metric tubular neighborhood of the zero section.
\end{definition}
\par The horizontal part of the metrics considered below interpolates between the base metric \(g\) and the Euclidean metric, while the vertical part is allowed to interpolate between another metric \(h\) on \(M\) and the Euclidean metric.
\begin{definition}[Interpolating metrics]
	Let \(\delta\) denote the Euclidean metric in the preferred coordinates near infinity. A Riemannian metric \(\widetilde g\) on \(TM\) is called an \emph{interpolating metric} if, on \(TU_\infty\),
	\begin{align*}
		\widetilde g_{(x,\y)}(X^h,Y^h)
		&=
		\bigl(
		\psi(x,\y)\,g_x
		+
		(1-\psi(x,\y))\,\delta
		\bigr)(X,Y),\\
		\widetilde g_{(x,\y)}(X^h,Y^v)
		&=
		a^c_{(x,\y)}(X,Y),\\
		\widetilde g_{(x,\y)}(X^v,Y^v)
		&=
		\bigl(
		\xi(x,\y)\,h_x
		+
		(1-\xi(x,\y))\,\delta
		\bigr)(X,Y),
	\end{align*}
	where \(g\) and \(h\) are Riemannian metrics on \(M\), \(a^c\) is a symmetric \((0,2)\)-tensor on \(TM\), and
	\[
	0\le\psi,\xi\le1
	\]
are smooth functions. We further assume that
	\[
	\psi(x,\y),\ \xi(x,\y)\longrightarrow0
	\qquad\text{as }|\y|\to\infty,
	\]
	and
	\[
	a^c_{(x,\y)}\longrightarrow0
	\qquad\text{as }|(x,\y)|\to\infty,
	\]
	with the latter convergence understood tensorially.
\end{definition}
\par The preceding general form simplifies considerably under the decay assumptions relevant to our construction.
\begin{lemma}\label{lem:inter-met}
	Suppose that \(\widetilde g\) is an interpolating metric and that
	\[
	\xi(x,\y)=o(|\y|^{-1})
	\qquad\text{as }|\y|\to\infty
	\]
	for every \(x\in U_\infty\). Then the connecting metric \(\mathdutchcal g\) is constant in the preferred coordinates on \(U_\infty\), equivalently,
	\[
	{}^{\mathdutchcal g}\Gamma_{ib}^{\,j}=0
	\qquad\text{on }U_\infty.
	\]
	Consequently, on \(TU_\infty\),
	\begin{align}\label{eq:nice-form}
		\widetilde g_{ij}
		&=
		\psi g_{ij}+(1-\psi)\delta_{ij},\notag\\
		\widetilde g_{i\dot j}
		&=
		a^c_{ij}, \\
		\widetilde g_{\dot i\dot j}
		&=
		\xi h_{ij}+(1-\xi)\delta_{ij}.\notag
	\end{align}
\end{lemma}
\begin{proof}
	Similar to the beginning of the proof of \hyperref[lem:g-gen-flat]{Lemma~\ref{lem:g-gen-flat}}, one infers
	\[
	N_i^{\,k}g^v_{kj}\equiv0.
	\]
	For an interpolating metric,
	\[
	g^v_{kj} =	\xi h_{kj}+(1-\xi)\delta_{kj},
	\]
	and hence
	\[
	\xi N_i^{\,k}h_{kj}
	+
	(1-\xi)N_i^{\,j}
	\equiv0.
	\]
	Since
	\[
	N_i^{\,k}(x,\y) = y^b\,{}^{\mathdutchcal g}\Gamma_{ib}^{\,k}(x) = \mathcal O(|\y|),
	\]
	while
	\[
	\xi(x,\y)= \smallO(|\y|^{-1}),
	\]
	we have
	\[
	\xi(x,\y)\,N_i^{\,k}(x,\y)\,h_{kj}(x)\longrightarrow0
	\]
	as \(|\y|\to\infty\). The preceding identity therefore implies
	\[
	N_i^{\,j}(x,\y)\longrightarrow0
	\qquad\text{as }|\y|\to\infty.
	\]
	But \(N_i^{\,j}\) is linear in \(\y\):
	\[
	N_i^{,j}(x,\lambda\y) = \lambda y^b\,{}^{\mathdutchcal g}\Gamma_{ib}^{\,j}(x).
	\]
	Thus, for every \(\y\neq0\),
	\[
	y^b\,{}^{\mathdutchcal g}\Gamma_{ib}^{\,j}(x)=0.
	\]
	Since this expression is linear in \(\y\),
	\[
	{}^{\mathdutchcal g}\Gamma_{ib}^{\,j}(x)=0
	\]
	for all \(i,b,j\). Hence the connecting metric is constant in the preferred coordinates on \(U_\infty\).
	
	The expression \hyperref[eq:nice-form]{\eqref{eq:nice-form}} now follows directly from the definition of an interpolating metric and the vanishing of the connection coefficients.
\end{proof}
\subsection{Admissible metrics}\label{subsubsec:adms-met}
\par We now impose the decay and monotonicity conditions that define the class of metrics for which the positive mass argument will be carried out.
\par Throughout this subsection, quantities carrying the subscript \(0\) are computed with respect to the Euclidean metric \(\delta\). On \(TU_\infty\), we write
\[
\rho^2=r^2+s^2,
\qquad
r:=|x|,
\qquad
s:=|\y|.
\]
\par The essential feature of the admissible class is that the horizontal component is permitted to exhibit the slow decay inherited from the base metric \(g\). The remaining components are required to decay sufficiently rapidly so that their contributions to the asymptotic flux are controlled.
\begin{definition}[Admissible metrics of type I]\label{defn:adms-1}
	A tubular interpolating metric \(\widetilde g\) is called an \emph{admissible metric of type I} if the following conditions hold.
	\begin{enumerate}
		\item\label{item:admI-1}
		The function
		\[
		\psi=\psi(\y) \not\equiv 0
		\]
		depends only on the fiber variable, is compactly supported, and satisfies
		\[
		\psi_{,s}\le0.
		\]
		Thus \(\psi\) is a radially non-increasing cutoff function in the fiber variable.
		\item\label{item:admI-2}
		The base metric, cross term, and vertical component satisfy
		\[
		g_{ij}-\delta_{ij}
		\in
		\mathcal O_2(r^{-(n-2)}),
		\]
		\[
		a^c_{ij}
		\in
		\mathcal O_1(\rho^{-\lambda^c}),
		\qquad \text{where} \quad 
		\lambda^c>2n-1,
		\]
		and
		\[
		h_{ij}
		\in
		\mathcal O_1(r^{-\sigma^v}),
		\qquad
		\xi\in\mathcal O_1(\rho^{-\tau^v}),
		\]
		where
		\[
		\sigma^v+\tau^v>2n-1,
		\qquad
		\tau^v\ge2n-1.
		\]
	\end{enumerate}
\end{definition}

According to the \hyperref[defn:AE-TB]{Definition~\ref{defn:AE-TB}}, equipped with admissible metrics of type I, the tangent bunle becomes an AF tnagent bundle with decay at least $n-2$.

\par The second class allows slower decay in the vertical interpolation, at the cost of imposing additional structure on the vertical metric.
\begin{definition}[Admissible metrics of type II]\label{defn:adms-2}
	A tubular interpolating metric \(\widetilde g\) is called an \emph{admissible metric of type II} if it satisfies \hyperref[item:admI-1]{condition~\ref{item:admI-1}} of \hyperref[defn:adms-1]{Definition~\ref{defn:adms-1}} and the following additional conditions.
	\begin{enumerate}
		\item\label{item:admII-1}
		The function
		\[
		\xi=\xi(x,\y)
		\]
		is compactly supported in the fiber variable and satisfies
		\[
		\xi_{,s}\le0.
		\]
		Thus, for each \(x\in U_\infty\), the function \(\xi(x,\cdot)\) is radially non-increasing cut-off function on the fiber.
		\item\label{item:admII-2}
		The vertical metric \(h\) is conformally flat:
		\[
		h_{ij}=e^{H(x)}\delta_{ij},
		\]
		where \(H\) is non-negative and radially non-increasing.
		\item\label{item:admII-3}
		The decay assumptions are
		\[
		g_{ij}-\delta_{ij}
		\in
		\mathcal O_2(r^{-(n-2)}),
		\]
		\[
		a^c_{ij}
		\in
		\mathcal O_1(\rho^{-\lambda^c}),
		\qquad \text{where} \quad 
		\lambda^c>2n-1,
		\]
		and
		\[
		H\in\mathcal O_1(r^{-\sigma^v}),
		\qquad
		\xi\in\mathcal O_1(\rho^{-\tau^v}),
		\]
		where
		\[
		\tau^v\ge1,
		\qquad
		\sigma^v+\tau^v\ge n-1.
		\]
	\end{enumerate}
\end{definition}
\par We refer to a metric as \emph{admissible} whenever it is admissible of either type I or type II.
\begin{definition}[Good metrics]
	An admissible metric \(\widetilde g\) is called \emph{good} if
	\[
	g=h, \qquad \text{and} \qquad a^c\equiv0.
	\]
Thus, for a good metric, the horizontal and vertical components are determined by the same metric on the base, and the horizontal--vertical cross term vanishes identically.
\end{definition}
\par The stronger decay imposed on the cross and vertical components will allow us, in the subsequent analysis, to isolate this contribution and relate the mass of $(TM,\widetilde g)$ to geometric quantities on \(M\).
 \section{Proof of the main theorem}\label{sec:proof-main}
 \par Throughout this section, all computations are carried out in a given preferred coordinates $\zeta$  on \(TM\) near infinity; recall that on $TU_\infty$, the latter coincides with
 \[
 (x^1,\ldots,x^n,y^1,\ldots,y^n)
 \]
We use the following notation.
\begin{itemize}
 	\item \(\widetilde\nu\) denotes the outward Euclidean unit normal to the coordinate sphere
 	\(\Sp_\rho\subset\mathbb R^{2n}\), while \(\nu_x\) denotes the outward Euclidean unit normal to the coordinate sphere \(\Sp_r\subset\mathbb R^n_x\). Thus
 	\[
 	\widetilde\nu^j=\frac{x^j}{\rho},
 	\qquad
 	\widetilde\nu^{\dot j}=\frac{y^j}{\rho}.
 	\]
 \item The Euclidean area measures on the spheres of radii \(r\), \(s\), and \(\rho\) are denoted by
 	\(\sigma_r\), \(\sigma_s\), and \(\sigma_\rho\), respectively.
 \item We write \(\mathbb R^n_x\) and \(\mathbb R^n_y\) for the Euclidean spaces with coordinates \(x^i\) and \(y^i\), respectively.
 \item The ADM flux integrand associated with a metric \(g\) is denoted by \(\mathscr I_g\).
 \item Repeated indices are summed according to the Einstein convention.
\end{itemize}
 \subsection{The flux decomposition}
\par By \hyperref[lem:inter-met]{Lemma~\ref{lem:inter-met}}, the connecting metric \(\mathdutchcal g\) is constant in the preferred coordinates on \(U_\infty\). In particular, its Levi--Civita connection vanishes there, and the metric \(\widetilde g\) takes the form
 \begin{equation}\label{eq:main-simple-form}
 	\widetilde g_{ij}=
 	\psi g_{ij}+(1-\psi)\delta_{ij},
 	\qquad
 	\widetilde g_{i\dot j}=a^c_{ij},
 	\qquad
 	\widetilde g_{\dot i\dot j}	=
 	\xi h_{ij}+(1-\xi)\delta_{ij}.
 \end{equation}
\par We evaluate the ADM flux of \(\widetilde g\) on the coordinate spheres
 \(\Sp_\rho\subset\mathbb R^{2n}\). 
 We write the ADM mass integrand $\mathscr I_{\widetilde g}$ (see \hyperref[eq:ADM]{\eqref{eq:ADM}}) as
\[
 \mathscr I_{\widetilde g}
 =
 \fancyRoman{1}
 +\fancyRoman{2}
 +\fancyRoman{3}
 +\fancyRoman{4},
 \]
 where the four terms correspond respectively to the horizontal-horizontal,
 mixed-horizontal, mixed-vertical, and vertical-vertical components.
 \par For the first contribution, \hyperref[eq:main-simple-form]{\eqref{eq:main-simple-form}} gives
 \begin{align}\label{eq:flux-1}
 	\fancyRoman{1}
 	&=
 	\left(
 	\widetilde g_{ij,i} - \widetilde g_{ii,j}
 	\right)\widetilde\nu^j \notag\\
 	&=
 	\psi
 	\left(
 	g_{ij,i}-g_{ii,j}
 	\right)
 	\frac{x^j}{\rho} \\
 	&=
 	\frac r\rho\,\psi\,\mathscr I_g.\notag
 \end{align}
\par For the second contribution,
 \begin{align}\label{eq:flux-II}
 	\fancyRoman{2}
 	&=
 	\left(
 	\widetilde g_{\dot i j,\dot i}
 	-
 	\widetilde g_{\dot i\dot i,j}
 	\right)\widetilde\nu^j \notag\\
 	&=
 	a^c_{ij,\dot i}\frac{x^j}{\rho}
 	-
 	\frac r\rho\,\xi_{,r}\,\left(\tr_0 h\right)
 	-
 	\frac r\rho\,\xi\left(\tr_0h\right)_{,r}.
 \end{align}
\par The third contribution is
 \begin{align}\label{eq:flux-III}
 	\fancyRoman{3}
 	&=
 	\left(
 	\widetilde g_{i\dot j,i}
 	-
 	\widetilde g_{ii,\dot j}
 	\right)\widetilde\nu^{\dot j}\notag\\
 	&=
 	a^c_{ij,i}\frac{y^j}{\rho}
 	-
 	\frac{s}{\rho}
 	\left(\tr_0g-n\right)\psi_{,s}.
 \end{align}
\par Finally,
 \begin{align}\label{eq:flux-IV}
 	\fancyRoman{4}
 	&=
 	\left(
 	\widetilde g_{\dot i\dot j,\dot i}
 	-
 	\widetilde g_{\dot i\dot i,\dot j}
 	\right)\widetilde\nu^{\dot j}\notag\\
 	&=
 	\xi_{,\dot i}\left(h_{ij}-\delta_{ij}\right)\frac{y^j}{\rho}
 	-
 	\frac{s}{\rho}\xi_{,s}\left(\tr_0h-n\right)\\
 	&=
 	\xi_{,\dot i}h_{ij}\frac{y^j}{\rho}
 	-
 	\frac{s}{\rho}\xi_{,s}
 	\bigl(\tr_0h-(n-1)\bigr).\notag
 \end{align}
\par The remainder of the proof consists of estimating these four contributions. The first term carries the ADM mass of the base manifold; the remaining terms are either nonnegative or asymptotically negligible under the admissibility hypotheses.
 \subsection{Admissible metrics of type I}
\par Suppose first that \(\widetilde g\) is admissible of type I.
\par Let
 \[
 F:TM\longrightarrow\mathbb R,
 \qquad
 F(x,\y)=|\y|=s.
 \]
 Since \(\psi\) is compactly supported in the fiber variable, choose \(R>0\) such that
 \[
 \supp(\psi)\subset B_R(0)\subset\mathbb R^n_y.
 \]
\par We begin with the horizontal contribution \(\fancyRoman{1}\) using \hyperref[eq:flux-1]{\eqref{eq:flux-1}}. Applying the co-area formula to the function \(F\) and using the fact that \(r^2+s^2=\rho^2\) on \(\Sp_\rho\), we obtain
 \begin{align*}
 	&\int_{\Sp_\rho}\fancyRoman{1}\,d\sigma_\rho \notag\\
 	&=
 	\int_0^R
 	\int_{\mathcal A_s^\rho}
 	\frac{\sqrt{\rho^2-s^2}}{\rho}
 	\psi(\y)
 	\int_{\Sp_{\sqrt{\rho^2-s^2}}\subset\mathbb R^n_x}
 	\mathscr I_g\,d\sigma_{\sqrt{\rho^2-s^2}}
 	\,d\sigma_s\,ds,
 \end{align*}
 where
 \[
 \mathcal A_s^\rho
 :=
 \left\{
 	\y\in\Sp_s \quad \text{\textbrokenbar}\quad
 	{\y}\times
 	\Sp_{\sqrt{\rho^2-s^2}}
 	\subset
 	TU_\infty\cap\Sp_\rho
 	\right\}.
 \]
\par Since \(M\) has a single asymptotic end, for every fixed \(R\) and all sufficiently large \(\rho\),
 \[
 \mathcal A_s^\rho=\Sp_s,
 \qquad
 0\le s\le R.
 \]
 Moreover,
 \[
 \frac{\sqrt{\rho^2-s^2}}{\rho}\longrightarrow1
 \]
 uniformly for \(0\le s\le R\). Since \(g\) has a well-defined ADM mass,
 \[
 \lim_{r\to\infty}
 \int_{\Sp_r}\mathscr I_g\,d\sigma_r
 =
 2(n-1)\omega_{n-1}\m(M,g).
 \]
 Therefore, for every fixed \(s\in[0,R]\),
 \begin{equation}\label{eq:base-flux-limit}
 	\lim_{\rho\to\infty}
 	\frac{\sqrt{\rho^2-s^2}}{\rho}
 	\int_{\Sp_{\sqrt{\rho^2-s^2}}}
 	\mathscr I_g\,d\sigma_{\sqrt{\rho^2-s^2}}
 	=
 	2(n-1)\omega_{n-1}\m(M,g).
 \end{equation}
\par Since the integrand is supported in \(0\le s\le R\), Fatou's lemma and
 \hyperref[eq:base-flux-limit]{\eqref{eq:base-flux-limit}} yield
 \begin{align*}
 	\varliminf_{\rho\to\infty}
 	\int_{\Sp_\rho}\fancyRoman{1}\,d\sigma_\rho
 	&\ge
 	2(n-1)\omega_{n-1}\m(M,g)
 	\int_0^R\int_{\Sp_s}\psi\,d\sigma_s\,ds\notag\\
 	&=
 	2(n-1)\omega_{n-1}
 	\|\psi\|_{L^1(\mathbb R^n_y)}
 	\m(M,g).
 \end{align*}
 \par Since 
 \[ 
 \rho \mapsto \int_{\Sp_{\sqrt{\rho^2-s^2}}} \mathscr I_g(x)\, d\sigma_{\sqrt{\rho^2-s^2}} \,d\sigma_s 
 \]
 is \emph{bounded above by a constant and the range of $s$ is bounded}, we can apply the Fatou's lemma for the upper limit (also known as the reverse Fatou);
 \par The same argument applied to the upper limit gives the reverse inequality,
 \begin{align*}
 	\varlimsup_{\rho\to\infty}
 	\int_{\Sp_\rho}\fancyRoman{1}\,d\sigma_\rho
 	&\le
 	2(n-1)\omega_{n-1}
 	\|\psi\|_{L^1(\mathbb R^n_y)}
 	\m(M,g).
 \end{align*}
 Consequently, 
 \begin{equation}\label{eq:flux-I-limit}
 	\lim_{\rho\to\infty}
 	\int_{\Sp_\rho}\fancyRoman{1}\,d\sigma_\rho
 	=
 	2(n-1)\omega_{n-1}
 	\|\psi\|_{L^1(\mathbb R^n_y)}
 	\m(M,g).
 \end{equation}
\par We next consider the second and fourth contributions. The decay assumptions
 \[
 a^c_{ij}\in\mathcal O_1(\rho^{-\lambda^c}),
 \qquad
 \lambda^c>2n-1,
 \]
 together with the corresponding assumptions on \(h\) and \(\xi\), imply
 \begin{equation}\label{eq:flux-II-small}
 	\int_{\Sp_\rho}\fancyRoman{2}\,d\sigma_\rho=\smallO(1),
 \end{equation}
 and
 \begin{equation}\label{eq:flux-IV-small}
 	\int_{\Sp_\rho}\fancyRoman{4}\,d\sigma_\rho= \smallO(1).
 \end{equation}
 In particular, both contributions vanish as
 \(\rho\to\infty\).
\par It remains to exploit the sign structure of \(\fancyRoman{3}\). From
 \hyperref[eq:flux-III]{\eqref{eq:flux-III}},
 \[
 \fancyRoman{3}
 =
a^c_{ij,i}\frac{y^j}{\rho}
 - 
\frac{s}{\rho}
 \left(\tr_0g-n\right)\psi_{,s}.
 \]
 The first term has vanishing asymptotic contribution by the decay of \(a^c\). Hence
 \begin{align*}
 	\varliminf_{\rho\to\infty}
 	\int_{\Sp_\rho}\fancyRoman{3}\,d\sigma_\rho
 	=
 	\varliminf_{\rho\to\infty}
 	\int_{\Sp_\rho}
 	-\frac{s}{\rho}
 	\left(\tr_0g-n\right)\psi_{,s}\,d\sigma_\rho.
 \end{align*}
\par Using the co-area formula once again,
 \begin{align}
 	&\int_{\Sp_\rho}
 	-\frac{s}{\rho}
 	\left(\tr_0g-n\right)\psi_{,s}\,d\sigma_\rho\notag\\
 	&=
 	\int_0^R
 	\int_{\Sp_s}
 	-\frac{s}{\rho}\psi_{,s}
 	\int_{\Sp_{\sqrt{\rho^2-s^2}}}
 	\left(\tr_0g-n\right)\,d\sigma_{\sqrt{\rho^2-s^2}}
 	\,d\sigma_s\,ds.
 	\label{eq:coarea-III}
 \end{align}
\par By the weak trace-majorizing assumption,
 \[
 \int_{\Sp_r}\left(\tr_0g-n\right)\,d\sigma_r\ge0
 \]
 for all sufficiently large \(r\). Since
 \[
 \psi_{,s}\le0,
 \]
 the integrand in \hyperref[eq:coarea-III]{\eqref{eq:coarea-III}} is nonnegative for sufficiently large \(\rho\). Thus
 \begin{equation}\label{eq:III-nonnegative}
 	\varliminf_{\rho\to\infty}
 	\int_{\Sp_\rho}\fancyRoman{3}\,d\sigma_\rho
 	\ge0.
 \end{equation}
\par We also need an upper bound. Since
 \[
 g_{ij}-\delta_{ij}
 =
 \mathcal O_2\left(r^{-(n-2)}\right),
 \]
 we have
 \[
 \tr_0g-n=\mathcal O\left(r^{-(n-2)}\right),
 \]
 and, for fixed \(s\),
 \[
 \frac{s}{\rho}=\mathcal O\left(r^{-1}\right).
 \]
 Hence
 \[
 \frac{s}{\rho}\left(\tr_0g-n\right)
 =
 \mathcal O\left(r^{-(n-1)}\right),
 \]
 so that
 \[
 \int_{\Sp_r} \frac{s}{\rho} \left( \tr_0g - n \right)\, d\sigma_r
 \]
remains uniformly bounded as $r\to\infty$. Since $\psi_{,s}$ is bounded and has compact support in the fiber variable, the resulting family of integrands admits a uniform integrable bound, which justifies the application of the reverse Fatou lemma.
\par By the co-area representation above, we thus get
 \begin{equation}\label{eq:III-bounded}
 	\varlimsup_{\rho\to\infty}
 	\int_{\Sp_\rho}\fancyRoman{3}\,d\sigma_\rho
 	<\infty.
 \end{equation}
\par Combining \hyperref[eq:III-nonnegative]{\eqref{eq:III-nonnegative}} and \hyperref[eq:III-bounded]{\eqref{eq:III-bounded}},
 \begin{equation}\label{eq:flux-III-bounded}
 	0\le
 	\varliminf_{\rho\to\infty}
 	\int_{\Sp_\rho}\fancyRoman{3}\,d\sigma_\rho
 	\le
 	\varlimsup_{\rho\to\infty}
 	\int_{\Sp_\rho}\fancyRoman{3}\,d\sigma_\rho
 	<\infty.
 \end{equation}
\par We can now combine the four contributions. By
 \hyperref[eq:flux-I-limit]{\eqref{eq:flux-I-limit}},
 \hyperref[eq:flux-II-small]{\eqref{eq:flux-II-small}},
 \hyperref[eq:flux-IV-small]{\eqref{eq:flux-IV-small}}, and
 \hyperref[eq:flux-III-bounded]{\eqref{eq:flux-III-bounded}},
 \begin{align}\label{eq:mass-lower-bound}
 	&2(2n-1)\omega_{2n-1}\,
 	\underline{\m}(TM,\widetilde g, \zeta)\notag\\
 	&\qquad=
 	\varliminf_{\rho\to\infty}
 	\int_{\Sp_\rho}\mathscr I_{\widetilde g}\,d\sigma_\rho\\
 	&\qquad\ge
 	2(n-1)\omega_{n-1}
 	\|\psi\|_{L^1(\mathbb R^n_y)}
 	\m(M,g).\notag
 \end{align}
 On the other hand, all four contributions have finite upper limits, and hence
 \begin{equation}\label{eq:mass-upper-finite}
 	\overline{\m}(TM,\widetilde g, \zeta)<\infty.
 \end{equation}
\par Since by the classical positive mass theorem applied to \((M,g)\),
 \[
 \m(M,g)\ge0,
 \]
 it follows that
 \begin{equation}\label{eq:mass-chain}
 	0
 	\le
 	\underline{\m}(TM,\widetilde g, \zeta)
 	\le
 	\overline{\m}(TM,\widetilde g, \zeta)
 	<\infty.
 \end{equation}
\par This proves the positivity and finiteness assertions in the type~I case.
\par Finally, suppose
 \[
 \underline{\m}(TM,\widetilde g, \zeta)=0.
 \]
 Since \(\psi\not\equiv0\), \hyperref[eq:mass-lower-bound]{\eqref{eq:mass-lower-bound}} implies
 \[
 \m(M,g)=0.
 \]
 The rigidity statement in the classical positive mass theorem, \hyperref[thm:PMTR]{Theorem~\ref{thm:PMTR}}, therefore gives
 \[
 (M,g)\cong_{\mathrm{isom}}\mathbb R^n,
 \]
 and, in the preferred coordinates,
 \[
 g=\delta.
 \]
 More precisely, the preceding estimates give
 \begin{align}\label{eq:key-mass-estimate}
 	&\frac{1}{2}	\varliminf_{\rho\to\infty}
 	\int_{\Sp_\rho}\fancyRoman{3}\,d\sigma_\rho\notag\\
 	&\le
 	(2n-1)\omega_{2n-1}
 	\underline{\m}(TM,\widetilde g, \zeta)
 	-
 	(n-1)\omega_{n-1}
 	\|\psi\|_{L^1(\mathbb R^n_y)}
 	\m(M,g)
 	\\
 	&\le
 	(2n-1)\omega_{2n-1}
 	\overline{\m}(TM,\widetilde g,\zeta)
 	-
 	(n-1)\omega_{n-1}
 	\|\psi\|_{L^1(\mathbb R^n_y)}
 	\m(M,g) \notag\\
 	&\le \frac{1}{2}	\varlimsup_{\rho\to\infty}
 	\int_{\Sp_\rho}\fancyRoman{3}\,d\sigma_\rho.\notag
 \end{align}
 If $g$ is weakly trace-rigid, then, one has
 \[
 \lim_{\rho\to\infty}
 \int_{\Sp_\rho}\fancyRoman{3}\,d\sigma_\rho = 0
 \]
thus \hyperref[eq:key-mass-estimate]{\eqref{eq:key-mass-estimate}} gives us
\[
\m(TM,\widetilde g, \zeta) =
\frac{(n-1)\omega_{n-1}}
{(2n-1)\omega_{2n-1}}
\m(M,g).
\]
\subsection{Admissible metrics of type II}
\par We now consider an admissible metric of type II. The contribution \(\fancyRoman{1}\) is unchanged, and hence
 \hyperref[eq:flux-I-limit]{\eqref{eq:flux-I-limit}} remains valid. Likewise, the mixed term involving \(a^c\) is controlled by the same decay assumption as in the type~I case.
\par The only new point is the structure of the vertical metric. Since
 \[
 h_{ij}=e^{H(x)}\delta_{ij},
 \]
 we have
 \[
 \tr_0h=ne^H.
 \]
 Consequently, \hyperref[eq:flux-II]{\eqref{eq:flux-II}} becomes
 \begin{align}
 	\fancyRoman{2}
 	&=
 	a^c_{ij,\dot i}\frac{x^j}{\rho}
 	-
 	n\frac r\rho\,\xi_{,r}\,e^H
 	n\frac r\rho\,\xi e^H H_{,r}.
 	\label{eq:type-II-II}
 \end{align}
\par Write
 \[
 \fancyRoman{2}'
 :=
 -n\frac r\rho
 \left(
 \xi_{,r}\,e^H+\xi e^HH_{,r}
 \right).
 \]
 Because
 \[
 \xi_{,r}\le0,
 \qquad
 H_{,r}\le0,
 \qquad
 \xi\ge0,
 \]
 we have
 \[
 \fancyRoman{2}'\ge0.
 \]
 The remaining term involving \(a^c\) is negligible by the assumed decay. Thus the type~II decay hypotheses imply
 \begin{equation}\label{eq:type-II-II-bound}
 	0\le
 	\varliminf_{\rho\to\infty}
 	\int_{\Sp_\rho}\fancyRoman{2}\,d\sigma_\rho
 	<\infty.
 \end{equation}
 Since on $[0,L]$, the function
 \[
 s\mapsto
 \int_{\Sp_r\subset\R^n_x}
 r^{-1}\xi_{,s}(1-e^H)\,d\sigma_r(x)
 \]
 is non-negative and uniformly bounded, applying the reverse Fatou, we get
 \[
 \varlimsup_{\rho\to\infty}
 \int_{\Sp_\rho}\fancyRoman{2}\,d\sigma_\rho
 <\infty.
 \]
\par The fourth contribution simplifies substantially. From
 \hyperref[eq:flux-IV]{\eqref{eq:flux-IV}},
 \[
 \fancyRoman{4}
 =
 \xi_{,\dot i}e^H\delta_{ij}\frac{y^j}{\rho}
  - \frac{s}{\rho}\xi_{,s}
 \bigl(ne^H-(n-1)\bigr).
 \]
 By the chain rule,
 \[
 \xi_{,\dot i}\frac{y^i}{\rho}
 =
 \frac{s}{\rho}\xi_{,s},
 \]
 and therefore
 \begin{align}
 	\fancyRoman{4}
 	&=
 	(n-1)\frac{s}{\rho}\xi_{,s}\,(1-e^H).
 	\label{eq:type-II-IV}
 \end{align}
\par Here
 \[
 \xi_{,s}\le0,
 \qquad
 1-e^H\le0,
 \]
 so the integrand in \hyperref[eq:type-II-IV]{\eqref{eq:type-II-IV}} is nonnegative.
\par Choose \(L>0\) so that
 \[
 \supp_y(\xi)\subset B_L(0).
 \]
 For fixed \(s\le L\),
 \[
 \frac{s}{\rho}=\mathcal O\left(r^{-1}\right),
 \qquad
 \text{as} \;\; r\to\infty.
 \]
 Moreover,
 \[
 1-e^H= \mathcal O(H),
 \]
 and the decay assumption on \(H\) gives
 \[
 H= \mathcal O\left(r^{-\sigma^v}\right).
 \]
 Together with the decay of \(\xi_{,s}\), this implies
 \[
 r^{-1}\xi_{,s}\left(1-e^H\right)
 =
 \mathcal O\left(r^{-n}\right)
 \]
 under the type~II admissibility conditions. Hence
 \[
 \int_{\Sp_r}
 r^{-1}\xi_{,s}\,\left(1-e^H\right)\,d\sigma_r
 =
  \mathcal O\left(r^{-1}\right),
 \]
 and consequently
 \begin{equation}\label{eq:type-II-IV-liminf}
 	\varliminf_{\rho\to\infty}
 	\int_{\Sp_\rho}\fancyRoman{4}\,d\sigma_\rho
 	\ge0.
 \end{equation}
\par The same decay estimate provides a uniform upper bound for the corresponding co-area integrals. Applying the upper-limit version of Fatou's lemma therefore gives
 \[
 \varlimsup_{\rho\to\infty}
 \int_{\Sp_\rho}\fancyRoman{4}\,d\sigma_\rho
 \le0.
 \]
 Thus
 \begin{equation}\label{eq:type-II-IV-limit}
 	\lim_{\rho\to\infty}
 	\int_{\Sp_\rho}\fancyRoman{4}\,d\sigma_\rho
 	=0.
 \end{equation}
 The remaining estimates are identical to those established in the type~I case; thus, so is the conclusing argument. This completes the proof of the main theorem.
 \qed
\begin{remark}[Optimality of the decay rate]\label{rem:opt-decay}
 	The decay threshold
 	\[
 	\sigma\ge n-2
 	\]
 	is intrinsic to the present argument. In particular, if the decay rate is strictly smaller than \(n-2\), the contribution
 	\[
 	\varliminf_{\rho\to\infty}
 	\int_{\Sp_\rho}\fancyRoman{3}\,d\sigma_\rho
 	\]
 	need not remain bounded. Thus the natural decay rate \(n-2\) is the weakest rate for which the present flux argument yields the required control.
 \end{remark}
\renewcommand{\baselinestretch}{1.2}
\bibliographystyle{plain}
\bibliography{MassTB} 
\vspace{10mm}
%
%
%
%
\end{document}